\documentclass[11pt]{article}
\usepackage{amsfonts}
\usepackage{mathrsfs}
\usepackage[dvips]{graphics}
\usepackage[dvips]{color}
\usepackage{amsthm}
\usepackage[]{amsmath}
\usepackage{CJK}
\usepackage{indentfirst}
\newtheorem{proposition}{Property}[section]
\newtheorem{lemma}[proposition]{Lemma}
\newtheorem{definition}[proposition]{Definition}
\newtheorem{theorem}[proposition]{Theorem}
\newtheorem{corollary}[proposition]{Corollary}

\newtheorem{property}[proposition]{Property}
\begin{document}
\begin{CJK*}{GBK}{song}

\centerline{\textbf{\Large{Envelope Words and Return Words Sequences}}}

\vspace{0.2cm}

\centerline{\textbf{\Large{in the Period-doubling Sequence}}}

\vspace{0.2cm}

\centerline{Huang Yuke\footnote[1]{School of Mathematics and Systems Science, Beihang University (BUAA), Beijing, 100191, P. R. China. E-mail address: huangyuke@buaa.edu.cn,~hyg03ster@163.com(Corresponding author).}
~~Wen Zhiying\footnote[2]{Department of Mathematical Sciences, Tsinghua University, Beijing, 100084, P. R. China. E-mail address: wenzy@tsinghua.edu.cn.}}

\vspace{1cm}

\noindent\textbf{Abstract}~~~
We consider the infinite one-sided sequence generated by the period-doubling substitution
$\sigma(a,b)=(ab,aa)$, denoted by $\mathbb{D}$.
Since $\mathbb{D}$ is uniformly recurrent, each factor $\omega$ appears infinite many times in the sequence, which is arranged as $\omega_p$ $(p\ge 1)$.
Let $r_p(\omega)$ be the $p$-th return word over $\omega$.
The main result is: for each factor $\omega$, the sequence $\{r_p(\omega)\}_{p\geq1}$ is $\Theta_1$ or $\Theta_2$, which are substitutive sequences and determined completely in this paper.

\vspace{0.2cm}


\vspace{0.2cm}

\noindent\textbf{2010 MR Subject Classification}~~~11B85; 68Q45

\section{Introduction}

The period-doubling sequence $\mathbb{D}$ has been heavily studied within mathematics and computer science, etc.
D.Damanik\cite{D2000} determined the number of palindromes (resp. $k$-th powers of words) of length $n$ occurring in $\mathbb{D}$ explicitly.
Sometimes, the period-doubling sequence is called the first difference of the Thue-Morse sequence. Here we use an equivalent substitution $\hat{\sigma}(1,0)=(10,11)$, and
the definition of the difference of an integer sequence is natural.
In 1998, Allouche-Peyri$\grave{e}$re-Wen-Wen\cite{APWW1998}
proved that all the Hankel determinants of the period-doubling sequence are odd integers.
In 2014, Guo-Wen\cite{GW2014} determined the automaticity of the Hankel determinants of difference sequences of the Thue-Morse sequence, including $\mathbb{D}$.
In 2015,
Parreau-Rigo-Rowland-Vandomme\cite{PRRV2015} proved that $\mathbb{D}$ have 2-abelian complexity
sequences that are 2-regular. Fu-Han\cite{FH2015} considered $t$-extensions of the Hankel determinants of some certain automatic sequences, such as $\mathbb{D}$.

F.Durand\cite{D1998} introduced the return words and proved that a sequence is primitive substitutive if and only if the set of its return words is finite.
L.Vuillon\cite{V2001} proved that an infinite word $\tau$ is a Sturmian sequence if and only if each nonempty factor $\omega\prec\tau$ has exactly two distinct return words.
But they have not studied the expressions of the return words and the properties of the sequence composed by the return words.
Let $\omega$ be a factor of $\mathbb{D}$. For $p\geq1$ and let $\omega_p=x_{i+1}\cdots x_{i+n}$ and $\omega_{p+1}=x_{j+1}\cdots x_{j+n}$.
The factor $x_{i+1}\cdots x_{j}$ is called the $p$-th return word of $\omega$ and denoted by $r_{p}(\omega)$.
If no confusion happens, we denote by $r_p$ for short.
The sequence $\{r_p(\omega)\}_{p\geq1}$ is called the return word sequence of factor $\omega$.
Denote $r_0(\omega)$ the prefix of $\mathbb{D}$ before $\omega_1$. It is not a return word.

Huang-Wen \cite{HW2015-1,HW2015-2} determine the structure of the return word sequences of the Fibonacci sequence $\mathbb{F}$ and the Tribonacci sequence $\mathbb{T}$, respectively.
More precisely, for any factor $\omega$ of $\mathbb{F}$ (resp. $\mathbb{T}$),
the return word sequence $\{r_p(\omega)\}_{p\geq1}$ is still $\mathbb{F}$ (resp. $\mathbb{T}$).
Using these properties, we determined the number of palindromes (resp. $k$-th powers of words) occurring in $\mathbb{F}[1,n]$ (the prefix of $\mathbb{F}$ of length $n$) and $\mathbb{T}[1,n]$ for all $n\geq1$, see \cite{HW2016-2,HW2016-3,HW2016-5}.
These topics are of great importance in computer science.

The main tool of the two papers is ``kernel word''.
However, in the studies of the period-doubling sequence $\mathbb{D}$, the techniques of kernel words fail. In fact, there is no kernel set in $\mathbb{D}$, which satisfies the ``uniqueness of kernel decomposition'', see \cite{HW2015-1,HW2015-2}.
To overcome this difficulty, we introduce a new notion called ``envelope word''.
\textbf{The main result of this paper is: for any factor $\omega$,
the return word sequence $\{r_p(\omega)\}_{p\geq1}$ is $\Theta_1=\tau_1(\mathbb{D})$ or $\Theta_2=\tau_2(\mathbb{D})$, where $\tau_1(a,b)=(a,bb)$, $\tau_2(a,b)=(ab,acac)$.}
In order to prove this property, we first determine the return word sequences of envelope words in Section 2; then we give two types of ``uniqueness of envelope extension'' in Section 3; using them, we determine the return word sequences of general factors in Section 4.

\vspace{0.2cm}

Let $\mathcal{A}=\{a,b\}$ be a binary alphabet.
We denote the concatenation of $\nu$ and $\omega$ by $\nu\omega$ or $\nu\cdot\omega$.
$\omega^k$ means the concatenation of $k$ factors $\omega$, called the $k$-th power factor of $\omega$.
The mirror word of $\omega$ is defined to be $\overleftarrow{\omega}=x_n\cdots x_2x_1$. A word $\omega$ is called a palindrome if $\omega=\overleftarrow{\omega}$.
We define $\overline{\cdot}:\mathcal{A}\rightarrow\mathcal{A}$ by $\overline{a}=b$ and $\overline{b}=a$.
Let $\rho=x_1\cdots x_n$ be a finite word.
For any $i\leq j\leq n$, we define $\rho[i,j]=x_ix_{i+1}\cdots x_{j-1}x_j$.
For convention, we denote $\rho[i]=\rho[i,i]=x_i$, $\rho[i,i-1]=\varepsilon$ (empty word).
We denote by $L(\omega,p)$ the position of the first letter of $\omega_p$.
We say that $\nu$ is a prefix (resp. suffix) of a word $\omega$ if there exists word $u$ such that $\omega=\nu u$ (resp. $\omega=u\nu$), $|u|\geq0$, which denoted by $\nu\triangleleft\omega$ (resp. $\nu\triangleright\omega$).
In this case, we write $\nu^{-1}\omega=u$ (resp. $\omega\nu^{-1}=u$), where $\nu^{-1}$ is the inverse word of $\nu$ such that $\nu\nu^{-1}=\nu^{-1}\nu=\varepsilon$.

The period-doubling sequence $\mathbb{D}$ is the fixed point beginning with $a$ of  substitution $\sigma(a,b)=(ab,aa)$.
We denote $A_m=\sigma^m(a)$ and $B_m=\sigma^m(b)$ for $m\geq0$.
Then $|A_m|=|B_m|=2^m$, where $|\omega|$ is the length of $\omega$.
Let $\delta_m\in\{a,b\}$ be the last letter of $A_m$. Obviously,
$\delta_m=a$ if and only if $m$ even; $\delta_{m+1}$ is the last letter of $B_m$.

For later use, we list some elementary properties of $\mathbb{D}$, which can be proved easily by induction:
$A_{m}\delta_m^{-1}=B_{m}\delta_{m+1}^{-1}$ and
$A_{m}=a\prod_{j=0}^{m-1} B_{j}$ for $m\geq0$.
Moreover $\mathbb{D}=a\prod_{j=0}^\infty B_j$.

\section{The return word sequences of envelope words in $\mathbb{D}$}

As we said in Section 1, there is no ``kernel set" in the period-doubling
sequence, which satisfies the ``uniqueness of kernel decomposition".
To overcome this difficulty, we introduce a new notion called ``envelope word".

\begin{definition}[Envelope words]\label{D2.1}
Let $\mathcal{E}=\{E_{i,m}:i=1,2,~m\geq1\}$ be a set of factors with
$$E_{1,m}=A_{m}\delta_{m}^{-1}\text{ and }E_{2,m}=B_{m}B_{m-1}\delta_{m}^{-1}.$$
We call $E_{i,m}$ the $m$-th envelope word of type $i$.
\end{definition}

\begin{definition}[]\label{D2.2}
Let $\Theta_1=\tau_1(\mathbb{D})$ and $\Theta_2=\tau_2(\mathbb{D})$, where $\tau_1(a,b)=(a,bb)$ and $\tau_2(a,b)=(ab,acac)$.
Obviously, $\Theta_1$ and $\Theta_2$ are $\mathbb{D}$ over the alphabets $\{a,bb\}$ and
$\{ab,acac\}$, respectively.
\end{definition}

In this section, we are going to prove the two theorems below:

\begin{theorem}\label{T2.1}\
The return word sequence $\{r_p(E_{1,m})\}_{p\geq1}$ is $\Theta_1$ over the alphabet
$$\{r_1(E_{1,m}),r_2(E_{1,m})\}=\{A_{m},A_{m-1}\}.$$
\end{theorem}

\begin{theorem}\label{T2.2}\
The return word sequence $\{r_p(E_{2,m})\}_{p\geq1}$ is $\Theta_2$ over the alphabet
$$\{r_1(E_{2,m}),r_2(E_{2,m}),r_4(E_{2,m})\}=\{A_{m-1},A_{m-1}A_{m}B_{m+1},B_{m}B_{m-1}\}.$$
\end{theorem}

\subsection{Prove of Theorem \ref{T2.1}}

We first give a \emph{criterion} to determine all occurrences of a factor, which is useful in our proofs.
Let $\omega=u\nu$ where $|u|,|\nu|>0$. Obviously, if $\mathbb{D}[L,L+|\omega|-1]=\omega$, then $\mathbb{D}[L,L+|u|-1]=u$.
Thus in order to determine all occurrences of $\omega$, we first find out all occurrences of $u$, then check whether these $u$'s can extend to $\omega$ (i.e. be followed by $\nu$) or not.

\begin{lemma}\label{L2.1}\
For $m\geq0$, $A_m$ occurs exactly twice in $A_mA_m$ (resp. $A_mB_mA_m$). More precisely, $A_m$ occurs as prefix and suffix.
\end{lemma}

\begin{proof} For $m=0$, $A_m=a$ occurs exactly twice in $A_mA_m=aa$ (resp. $A_mB_mA_m=aba$) as prefix and suffix.
Assume the conclusions hold for $m$, consider $A_{m+1}=A_{m}B_{m}$ occurs in $$A_{m+1}A_{m+1}=\underbrace{A_{m}}_{[1]}B_{m}\underbrace{A_{m}}_{[2]}B_{m},
~A_{m+1}B_{m+1}A_{m+1}=\underbrace{A_{m}}_{[3]}B_{m}\underbrace{A_{m}}_{[4]} \underbrace{A_{m}}_{[5]}\underbrace{A_{m}}_{[6]}B_{m}.$$

Since $A_m$ occurs exactly twice in $A_mA_m$ and $A_mB_mA_m$, all possible positions of the prefix $A_{m}$ of $A_{m+1}$ are shown with ``underbrace'' above.

(1) The $A_{m}$ locates at [1], [2], [3], [6] followed by $B_m$, so they can extend to a $A_{m+1}=A_mB_m$.

(2) The $A_{m}$ locates at [4], [5] followed by $A_m\neq B_m$, so they can't extend to a $A_{m+1}=A_mB_m$.

This means, $A_{m+1}$ occurs exactly twice in $A_{m+1}A_{m+1}$ (resp. $A_{m+1}B_{m+1}A_{m+1}$) as prefix and suffix.
By induction, the conclusions hold for $m\geq0$.
\end{proof}

\begin{property}[]\label{P2.4}\
$r_0(E_{1,m})=\varepsilon$,
$r_1(E_{1,m})=A_{m}$ and
$r_2(E_{1,m})=A_{m-1}$ for $m\geq1$.
\end{property}

\begin{proof} We only need to determine the first three positions of $E_{1,m}$, denoted by $L(E_{1,m},p)$, $p=1,2,3$. By the criterion above, we determine $L(A_{m-1},p)$ first, $p=1,2,3$. Since
$$A_{m+2}=A_{m-1}B_{m-1}A_{m-1}A_{m-1}A_{m-1}B_{m-1}A_{m-1}B_{m-1}\triangleleft\mathbb{D},$$
By Lemma \ref{L2.1}, we have
$L(A_{m-1},1)=1$, $L(A_{m-1},2)=2^m+1$, $L(A_{m-1},3)=3\times2^{m-1}+1$.
Since $A_{m}\delta_m^{-1}=B_{m}\delta_{m+1}^{-1}$ , all of them are followed by $B_{m-1}\delta_{m}^{-1}$, i.e., can extend to $E_{1,m}$.

This means, $L(E_{1,m},1)=1$, $L(E_{1,m},2)=2^m+1$, $L(E_{1,m},3)=3\times2^{m-1}+1$.

By the definition of return words $r_p(E_{1,m})$, we have

$r_0(E_{1,m})=\mathbb{D}[1,L(E_{1,m},1)-1]=\mathbb{D}[1,0]=\varepsilon$;

$r_1(E_{1,m})=\mathbb{D}[L(E_{1,m},1),L(E_{1,m},2)-1]=\mathbb{D}[1,2^m]=A_{m}$;

$r_2(E_{1,m})=\mathbb{D}[L(E_{1,m},2),L(E_{1,m},3)-1]=\mathbb{D}[2^m+1,3\times2^{m-1}]=A_{m-1}$.
\end{proof}

\begin{corollary}\label{C2.1}
$|r_0(E_{1,m})|=0$,
$|r_1(E_{1,m})|=2^{m}$ and
$|r_2(E_{1,m})|=2^{m-1}$ for $m\geq1$.
\end{corollary}

Let $m\geq1$ be fixed. Define an alphabet
$$\mathcal{A}_{1,m}=\{\mathbf{a},\mathbf{b}\}=\{r_1(E_{1,m}),r_2(E_{1,m})r_2(E_{1,m})\}
=\{A_{m},B_{m}\}.$$

We denote
$\mathbf{A_n}(\mathcal{A}_{1,m})$ and $\mathbf{B_n}(\mathcal{A}_{1,m})$
the $A_n$ and $B_n$ over alphabet $\mathcal{A}_{1,m}$ for each fixed $m\geq1$.
If no confusion happens, we simply write $\mathbf{A_n}$, $\mathbf{B_n}$, $r_1$ and $r_2$ for short.

\begin{lemma}\label{L2.2}
Over $\mathcal{A}_{1,m}=\{\mathbf{a},\mathbf{b}\}$, 
$\mathbf{A_n}=A_{m+n}$ and $\mathbf{B_n}=B_{m+n}$ for $n\geq0$.
\end{lemma}

\begin{proof}
Since $\mathbf{A_0}=\mathbf{a}=A_{m}$, $\mathbf{B_0}=\mathbf{b}=B_{m}$, the conclusions hold for $n=0$. Assume they hold for $n$,
then $\mathbf{A_{n+1}}=\mathbf{A_nB_n}=A_{m+n}B_{m+n}=A_{m+n+1},~
\mathbf{B_{n+1}}=\mathbf{A_nA_n}=A_{m+n}A_{m+n}=B_{m+n+1}.$
By induction, the conclusions hold for $n\geq0$.
\end{proof}

By the definition of alphabet $\mathcal{A}_{1,m}$ and Definition \ref{D2.2}, Theorem \ref{T2.1} is equivalent to

\vspace{0.4cm}

\noindent\textbf{Theorem \ref{T2.1}'.}
\emph{The return word sequence $\{r_p(E_{1,m})\}_{p\geq1}$ in $\mathbb{D}$ is $\mathbb{D}$ over the alphabet $\mathcal{A}_{1,m}$.}

\begin{proof} By $A_{m}=a\prod_{j=0}^{m-1}B_j$ and Lemma \ref{L2.2}, we have
$$\begin{array}{c}
\mathbb{D}=a\prod_{j=0}^\infty B_j=a\prod_{j=0}^{m-1}B_j\prod_{j=m}^\infty B_j
=\mathbf{a}\prod_{j=0}^\infty B_{m+j}=\mathbf{a}\prod_{j=0}^\infty \mathbf{B_{j}}
=\mathbf{D}.
\end{array}$$
Here $\mathbb{D}$ and $\mathbf{D}$ are the period-doubling sequence over alphabets
$\mathcal{A}$ and $\mathcal{A}_{1,m}$, respectively.

Thus the conclusion holds.
\end{proof}

\subsection{Prove of Theorem \ref{T2.2}}

\begin{lemma}\label{L2.3}

(1) $B_m$ occurs exactly once in $A_mB_m$, at position $2^{m}+1$ for $m\geq0$;

(2) $B_m$ occurs three times in $B_mA_mB_m$, at positions 1, $2^{m-1}+1$ and $2^{m+1}+1$ for $m\geq1$;

(3) $B_m$ occurs three times in $B_mA_mA_mA_mB_m$, at positions 1, $2^{m-1}+1$ and $2^{m+2}+1$ for $m\geq1$.
\end{lemma}

Similarly as Property \ref{P2.4} and by Lemma \ref{L2.3}, the first five positions of $E_{2,m}$ are:
$2^m+1,~3\times 2^{m-1}+1,~5\times 2^m+1,~11\times2^{m-1}+1,~7\times 2^m+1.$
The expressions of $r_p(E_{2,m})$, $1\leq p\leq4$ are

\begin{property}[]\label{P2.5}
$r_0(E_{2,m})=A_{m}$, $r_1(E_{2,m})=r_3(E_{2,m})=A_{m-1}$, $r_2(E_{2,m})=A_{m-1}A_{m}B_{m+1}$ and $r_4(E_{2,m})=B_{m}B_{m-1}$ for $m\geq1$.
\end{property}

\begin{corollary}\label{C2.2}
$|r_0(E_{2,m})|=2^{m}$,
$|r_1(E_{2,m})|=2^{m-1}$,
$|r_2(E_{2,m})|=7\ast2^{m-1}$,
$|r_4(E_{2,m})|=3\ast2^{m-1}$.
\end{corollary}

Let $m\geq1$ be fixed. We simply write $r_p(E_{2,m})$ as $r_p$ for short, $p=1,2,4$.
Define an alphabet
$$\mathcal{A}_{2,m}=\{\mathbf{a},\mathbf{b}\}=
\{r_1r_2,r_1r_4r_1r_4\}
=\{A^{-1}_{m}A_{m+2}A_{m},A^{-1}_{m}B_{m+2}A_{m}\}.$$

We denote
$\mathbf{A_n}(\mathcal{A}_{2,m})$ and $\mathbf{B_n}(\mathcal{A}_{2,m})$
the $A_n$ and $B_n$ over alphabet $\mathcal{A}_{2,m}$ for each fixed $m\geq1$.
If no confusion happens, we simply write $\mathbf{A_n}$ and $\mathbf{B_n}$ for short.

By induction as Lemma \ref{L2.2}, we have

\begin{lemma}\label{L2.4}
Over $\mathcal{A}_{2,m}=\{\mathbf{a},\mathbf{b}\}$,
$\mathbf{A_n}=A^{-1}_{m}A_{m+n+2}A_{m}$ and
$\mathbf{B_n}=A^{-1}_{m}B_{m+n+2}A_{m}$ for $n\geq0$.
\end{lemma}

By the definition of alphabet $\mathcal{A}_{2,m}$ and Definition \ref{D2.2}, Theorem \ref{T2.2} is equivalent to

\vspace{0.4cm}

\noindent\textbf{Theorem \ref{T2.2}'.}
\emph{The return word sequence $\{r_p(E_{2,m})\}_{p\geq1}$ in $\mathbb{D}$ is $\mathbb{D}$ over the alphabet $\mathcal{A}_{2,m}$.}

\begin{proof} By $A_{m}=a\prod_{j=0}^{m-1} B_{j}$ and Lemma \ref{L2.4}, we have
$$\begin{array}{rl}
\mathbb{D}=&a\prod_{j=0}^\infty B_j=a\prod_{j=0}^{m+1}B_j\prod_{j=m+2}^\infty B_j
=A_{m+2}\prod_{j=0}^\infty B_{m+j+2}\\
=&A_{m} A^{-1}_{m}A_{m+2}A_{m} \prod_{j=0}^\infty (A^{-1}_{m}B_{m+j+2}A_{m})
=r_0(E_{2,m})\mathbf{a}\prod_{j=0}^\infty \mathbf{B_{j}}
=r_0(E_{2,m})\mathbf{D}.
\end{array}$$
Here $\mathbb{D}$ and $\mathbf{D}$ are the period-doubling sequence over alphabets
$\mathcal{A}$ and $\mathcal{A}_{2,m}$, respectively.
Notice that, by the definition of return word sequence,
we omit $r_0(E_{2,m})$. So the conclusion holds.
\end{proof}

\section{Uniqueness of envelope extension}

In this section, we give the two types of \emph{uniqueness of envelope extension}, which play an important role in our studies.
Using them, we can extend Theorem \ref{T2.1}, \ref{T2.2} and other related properties from envelope words to general factors.

\begin{definition}[The order of envelope words]\label{D3.1}

$E_{1,m}\sqsubset E_{2,m}$,
$E_{i,m}\sqsubset E_{j,m+1}$ for $i,j\in\{1,2\}$, $m\geq1$.
\end{definition}

\begin{definition}[Envelope of factor $\omega$, $Env(\omega)$]\label{D3.2}\
For any factor $\omega$, let
$$Env(\omega)=\min_{\sqsubset}\{E_{i,m}:\omega\prec E_{i,m},~i=1,2,~m\geq1\},$$
which is called the envelope of factor $\omega$.
\end{definition}

For any $\omega$, denote $Env(\omega)$ by $E_{i,m}$.
By Definition \ref{D3.2}, we know the envelope of factor $\omega$ is unique.
But we don't know: (1) whether $\omega$ occurs in $Env(\omega)$ only once or not;
(2) the relation between $\{\omega_p\}_{p\geq1}$ and $\{E_{i,m,p}\}_{p\geq1}$, i.e., the relation between positions $L(\omega,p)$ and $L(E_{i,m},p)$ for all $p\geq1$.

The two types of uniqueness of envelope extension will answer the two questions:

$\bullet$ The weak type (see Definition \ref{T3.2} below) shows the relation between $\omega$ and $Env(\omega)$;

$\bullet$ The strong type (see Definition \ref{T3.8} below) shows the relation between $\omega_p$ and $Env(\omega)_p$ for each $p\ge 1$, i.e.
$Env(\omega_p)=Env(\omega)_p$.

\subsection{Basic properties of envelope words}

For later use, we give some basic properties of envelope words first. By the definition of envelope words, $E_{1,m}=A_m\delta_{m}^{-1}$ and $E_{2,m}=B_{m}B_{m-1}\delta_{m}^{-1}$. Thus

\begin{property}\label{P3.1}
$E_{1,m+1}=E_{1,m}\delta_{m}E_{1,m}$ and $E_{2,m+1}=E_{1,m}\delta_{m}E_{1,m}\delta_{m}E_{1,m}$ for $m\geq1$.
\end{property}

\noindent\emph{Example.}  When $m=2$, $E_{1,m}=aba$, $\delta_{m}=a$. We have
$E_{1,m+1}=aba\underline{a}aba$, $E_{2,m+1}=aba\underline{a}aba\underline{a}aba$.
Here we give all $\delta_{m}=a$ with underline.

\begin{corollary}\label{C3.1}\
Both $E_{1,m}$ and $E_{2,m}$ are palindromes.
\end{corollary}

By induction, we can give a more general form of Property \ref{P3.1}, as Property \ref{P3.2} and \ref{P3.3}.

\begin{property}[Relation between $E_{1,m}$ and $E_{1,n}$]\label{P3.2}
For $m>n$,
$E_{1,m}=E_{1,n}x_1E_{1,n}\cdots E_{1,n}x_hE_{1,n}$,
where $x_1\cdots x_h=\overline{E_{1,m-n}}$ for $n$ odd, and
$x_1\cdots x_h=E_{1,m-n}$ for $n$ even.
\end{property}

\noindent\emph{Example.}  When $n=1$ (odd) and $m=3$, we give all $E_{1,1}=a$ with underline: $E_{1,3}=\underline{a}b\underline{a}a\underline{a}b\underline{a}.$
In this case, $x_1\cdots x_h=bab=\overline{aba}=\overline{E_{1,2}}=\overline{E_{1,m-n}}$.

\begin{property}[Relation between $E_{2,m}$ and $E_{1,n}$]\label{P3.3}
For $m>n$,
$E_{2,m}=E_{1,n}x_1E_{1,n}\cdots E_{1,n}x_hE_{1,n}$,
where $x_1\cdots x_h=\overline{E_{2,m-n}}$ for $n$ odd, and
$x_1\cdots x_h=E_{2,m-n}$ for $n$ even.
\end{property}

\noindent\emph{Example.}  When $n=1$ (odd) and $m=3$, we give all $E_{1,1}=a$ with underline: $E_{2,3}=\underline{a}b\underline{a}a\underline{a}b\underline{a}a\underline{a}b\underline{a}.$
In this case, $x_1\cdots x_h=babab=\overline{ababa}=\overline{E_{2,2}}=\overline{E_{2,m-n}}$.

\vspace{0.2cm}

In Corollary \ref{C3.1}, we have all envelope words are palindromes. Property \ref{P3.4} and \ref{P3.5} show stronger relations between envelope words and palindromes.
We first give some lemmas.

Since $E_{1,m}=A_m\delta_m^{-1}$, it is easy to proof that

\begin{lemma}[]\label{L3.2}
For $m\geq2$, we denote $E_{1,m}=x_1x_2\cdots x_h$, then $x_1=a$, $x_2=b$, $x_h=a$.
\end{lemma}

\begin{lemma}[]\label{L1.3}
The factor $E_{1,n}\delta_{m}E_{1,k}$ be a palindrome has only two cases:
(a) $n=k$.

(b) $|n-k|=1$; $m$ and $\min\{n,k\}$ have the same parity.
\end{lemma}

\begin{proof} Obviously, when $n=k$, $E_{1,n}\delta_{m}E_{1,k}$ is a palindrome.
When $n\neq k$, since both $E_{1,n}$ and $E_{1,k}$ are palindrome, we assume $n>k$ without loss of generality.
(1) When $n=k+1$, by Property \ref{P3.2}, $E_{1,n}\delta_{m}E_{1,k}=E_{1,k}\underline{\delta_{k}}E_{1,k}\underline{\delta_{m}}E_{1,k}$.
Comparing the two letters with underlines, $E_{1,n}\delta_{m}E_{1,k}$
is a palindrome if and only if $\delta_{k}=\delta_{m}$, i.e., $m$ and $k=\min\{n,k\}$ have the same parity.

(2) When $n\geq k+2$, by Property \ref{P3.2},
$$E_{1,n}\delta_{m}E_{1,k}=E_{1,k}x_1E_{1,k}\underline{x_2}E_{1,k}\cdots E_{1,k}\underline{x_h}E_{1,k}\delta_{m}E_{1,k},$$
where $x_1x_2\cdots x_h=E_{1,n-k}$ for $k$ odd, and $x_1x_2\cdots x_h=\overline{E_{1,n-k}}$ for $k$ even. No matter $k$ is odd or even, $x_2\neq x_h$ by Lemma \ref{L3.2}.
Thus $E_{1,n}\delta_{m}E_{1,k}$ can not be a palindrome in this case.

In summary, the conclusion holds.
\end{proof}

\begin{lemma}[]\label{L1.4}
For $n,h<m$, the factor $E_{1,n}\delta_{m}E_{1,m}\delta_{m}E_{1,k}$ be a palindrome if and only if $n=k$.
\end{lemma}

\begin{proof} Obviously, when $n=k$, $E_{1,n}\delta_{m}E_{1,m}\delta_{m}E_{1,k}$ is palindrome.
When $n\neq k$, since both $E_{1,n}$ and $E_{1,k}$ are palindrome, we assume $n>k$ without loss of generality.
(1) When $n=k+1$, by Property \ref{P3.2}, $E_{1,n}=E_{1,k}\delta_{k}E_{1,k}$ and
$E_{1,m}=E_{1,k}x_1E_{1,k}x_2E_{1,k}\ldots E_{1,k}x_hE_{1,k}$ for $h\geq3$, i.e.,
\begin{equation*}
\begin{split}
&E_{1,n}\cdot\delta_{m}\cdot E_{1,m}\cdot\delta_{m}\cdot E_{1,k}\\
=&E_{1,k}\underbrace{\delta_{k}}_{[1]}E_{1,k}\cdot\underbrace{\delta_{m}}_{[3]}\cdot E_{1,k}x_1E_{1,k}x_2E_{1,k}\ldots E_{1,k}x_{h-1}E_{1,k}x_hE_{1,k}\cdot\underbrace{\delta_{m}}_{[2]}\cdot E_{1,k}.
\end{split}
\end{equation*}

Suppose $E_{1,n}\delta_{m}E_{1,m}\delta_{m}E_{1,k}$ is palindrome, then
$\delta_{k}$ locates at [1] is equal to $\delta_{m}$ locates at [2], i.e. $\delta_{k}=\delta_{m}$.
Moreover $\delta_{m}$ locates at [3] is equal to $x_h$, i.e. $\delta_{m}=x_h$. Since $E_{1,m}$ is palindrome, $x_1=x_h$.
Since $E_{1,n}\delta_{m}E_{1,m}\delta_{m}E_{1,k}$ is palindrome, $x_1=x_{h-1}$.
(1) If $m=k+2$, i.e., $h=3$, we find a 6-th power factor of $\alpha\in\{a,b\}$ in $\Theta_1$ that $\delta_{k}\delta_{m}x_1x_2x_h\delta_{m}$.
(2) If $m\geq k+3$, $h\geq7$. Since $E_{1,m}$ is palindrome, $x_2=x_{h-1}$. We find a 4-th power factor of $\alpha\in\{a,b\}$ in $\Theta_1$ that $\delta_{k}\delta_{m}x_1x_2$.

Both of the two cases above contradict that there is no $k$-th power factor in $\Theta_1$, $k\geq4$. Thus $E_{1,n}\delta_{m}E_{1,m}\delta_{m}E_{1,k}$ can not be a palindrome for $n\neq k$. So the conclusion holds.
\end{proof}

\begin{property}[]\label{P3.4}
Let palindrome $\omega$ be a prefix (resp. suffix) of $E_{1,m}$, then there exists $n$ ($\leq m$) such that
$\omega=E_{1,n}$.
\end{property}

\begin{proof} For $m=1$, $E_{1,m}=a$, the palindromical prefix is $a$. For $m=2$, $E_{1,m}=aba$, the palindromical prefixes are $a$ and $aba$. The conclusion holds.
By induction, we assume the conclusions hold for $m$.
If $\omega=E_{1,m+1}$, the conclusion holds obviously.
If $\omega$ is a proper prefix of $E_{1,m+1}=E_{1,m}\delta_{m}E_{1,m}$, then

Case 1: $\omega$ is a prefix of $E_{1,m}$, then there exists $n\leq m$ such that
$\omega=E_{1,n}$.

Case 2: $\omega$ isn't a prefix of $E_{1,m}$. Since $E_{1,m+1}=E_{1,m}\delta_{m}E_{1,m}$, $\omega=E_{1,m}\delta_{m}\nu$, where $\nu$ is a prefix of $E_{1,m}$. Since $\omega$ is a palindrome,
$\omega=\overleftarrow{\omega}=\overleftarrow{\nu}\delta_{m}E_{1,m}$. This means $\overleftarrow{\nu}$ is also the prefix of $E_{1,m}$. So
$\nu=\overleftarrow{\nu}$ is a palindrome.
By assumption and $\omega\neq E_{1,m+1}$, there exists $n<m$ s.t. $\nu=E_{1,n}$. This means $\omega=E_{1,m}\delta_{m}E_{1,n}$, $m\neq n$. By Lemma \ref{L1.3}, $\omega$ can't be a palindrome, contradicting $\omega=\overleftarrow{\omega}$.

By the two cases above and by induction, if palindrome $\omega$ is a prefix of $E_{1,m}$, the conclusions hold for all $m$ .
Similarly, if palindrome $\omega$ is a suffix of $E_{1,m}$, there exists $n$ ($\leq m$) such that $\omega=E_{1,n}$.
\end{proof}

\begin{property}[]\label{P3.5}\
Let palindrome $\omega$ be a proper prefix (resp. suffix) of $E_{2,m}$, then there exists $n$ ($\leq m$) such that $\omega=E_{1,n}$.
\end{property}

\begin{proof}For $m=1$, $E_{2,m}=aa$, the palindromical proper prefix is $a$. For $m=2$, $E_{2,m}=ababa$, the palindromical prefixes are $a$ and $aba$. The conclusion holds.
Assume the conclusions hold for $m$.
If palindrome $\omega$ is a proper prefix of $E_{2,m+1}=E_{1,m+1}\delta_{m}E_{1,m}$, then

Case 1: $\omega$ is a prefix of $E_{1,m+1}$, then there exists $n\leq m+1$ such that
$\omega=E_{1,n}$.

Case 2: $\omega$ isn't a prefix of $E_{1,m+1}$, then $\omega=E_{1,m+1}\delta_{m}\nu$, where $\nu$ is a proper prefix of $E_{1,m}$. Since $\omega$ is a palindrome,
$\omega=\overleftarrow{\omega}=\overleftarrow{\nu}\delta_{m}E_{1,m+1}$. So $\overleftarrow{\nu}$ is the prefix of $E_{1,m+2}$. Moreover $|\nu|<|E_{1,m}|$, so
$\overleftarrow{\nu}$ is the prefix of $E_{1,m}$.
This means $\nu=\overleftarrow{\nu}$ is a palindrome.
By Property \ref{P3.4}, there exists $n< m$ s.t. $\nu=E_{1,n}$. This means $\omega=E_{1,m+1}\delta_{m}E_{1,n}$, $n\neq m+1$. By Lemma \ref{L1.3}, $\omega$ can't be a palindrome, contradicting $\omega=\overleftarrow{\omega}$.

By the two cases above and by induction, if palindrome $\omega$ is a prefix of $E_{2,m}$, the conclusions hold for all $m$ .
Similarly, if palindrome $\omega$ is a suffix of $E_{2,m}$, there exists $n$ ($\leq m$) such that $\omega=E_{1,n}$.
\end{proof}

\subsection{The simplification by palindromic property}

The ``weak type of envelope extension'' means ``each $\omega$ occurs in $Env(\omega)$ only once''.
Though the analysis in this subsection, we only need to prove the last property for $\omega$ is a palindrome.
Moreover, if $\omega$ occurs in $Env(\omega)$ at least twice, we can pick two of them.
So we only need to negate the proposition that ``there exist a palindrome $\omega$ occurs in $Env(\omega)$ twice''.

\begin{lemma}[]\label{L3.1}
For palindrome $\Lambda$, $|\Lambda|$ is odd, if there exist factor $\omega$ satisfies:

(1) $\omega$ occurs in $\Lambda$ twice;

(2) Both of the two $\omega$'s contain the middle letter of $\Lambda$.

Then there exist a palindrome $\varpi$ occurs in $\Lambda$ twice, at symmetric positions.
\end{lemma}

\begin{proof} Denote the two $\omega$ in $\Lambda$ by $\omega_1$ and $\omega_2$, and
$\Lambda=\alpha\omega_1\beta=\gamma\omega_2\eta,$
where $|\alpha|<|\gamma|$. Denote $|\alpha|:=d_1$ and $|\eta|:=d_2$.
Since both of $\omega_1$ and $\omega_2$ contain the middle letter of $\Lambda$, they are overlapped.
\setlength{\unitlength}{1mm}
\begin{center}
\begin{picture}(135,48)
\linethickness{1.5pt}
\put(18,0){(1) $d_1=d_2$}
\put(0,20){\line(1,0){50}}
\put(52,19){$\overleftarrow{\Lambda}=\Lambda$}
\put(5,23){\line(1,0){25}}
\put(32,22){$\overleftarrow{\omega_2}$}
\put(20,26){\line(1,0){25}}
\put(47,25){$\overleftarrow{\omega_1}$}
\put(25,29){$\Downarrow$ mirror}
\put(0,35){\line(1,0){50}}
\put(52,34){$\Lambda$}
\put(5,38){\line(1,0){25}}
\put(32,37){$\omega_1$}
\put(20,41){\line(1,0){25}}
\put(47,40){$\omega_2$}
\linethickness{0.5pt}
\put(0,38){\line(0,-1){3}}
\put(5,38){\line(0,-1){3}}
\put(1,36){$d_1$}
\put(45,41){\line(0,-1){6}}
\put(50,38){\line(0,-1){3}}
\put(46,36){$d_2$}
\linethickness{1.5pt}
\put(88,0){(2) $d_1<d_2$}
\put(70,5){\line(1,0){50}}
\put(122,4){$\Lambda$}
\put(75,8){\line(1,0){35}}
\put(112,7){$\varpi_1$}
\put(80,11){\line(1,0){35}}
\put(117,10){$\varpi_2$}
\put(95,14){$\Downarrow$ unite}
\put(70,20){\line(1,0){50}}
\put(122,19){$\overleftarrow{\Lambda}=\Lambda$}
\put(80,23){\line(1,0){20}}
\put(102,22){$\overleftarrow{\omega_2}$}
\put(95,26){\line(1,0){20}}
\put(117,25){$\overleftarrow{\omega_1}$}
\put(95,29){$\Downarrow$ mirror}
\put(70,35){\line(1,0){50}}
\put(122,34){$\Lambda$}
\put(75,38){\line(1,0){20}}
\put(97,37){$\omega_1$}
\put(90,41){\line(1,0){20}}
\put(112,40){$\omega_2$}
\linethickness{0.5pt}
\put(70,38){\line(0,-1){3}}
\put(75,38){\line(0,-1){30}}
\put(71,36){$d_1$}
\put(110,41){\line(0,-1){33}}
\put(120,38){\line(0,-1){3}}
\put(114,36){$d_2$}
\put(80,23){\line(0,-1){12}}
\put(115,26){\line(0,-1){15}}
\put(82,39){$\nu_1$}
\put(91,42){$\nu_2$}
\put(100,42){$\nu_3$}
\end{picture}
\end{center}
\centerline{Fig 3.1: A palindrome $\varpi$ occurs in $\Lambda$ twice, at symmetric positions.}

\vspace{0.2cm}

Case 1. $d_1=d_2$, see Fig 3.1(1).
In this case, $\omega_1=\overleftarrow{\omega_2}$, i.e., $\omega=\overleftarrow{\omega}$,
$\omega$ is a palindrome.

Case 2. $d_1\neq d_2$, let's just take $d_1<d_2$, see Fig 3.1(2).

In this case, let $\omega_1$ and $\omega_2$ overlapped at $\nu_2$, then $\omega_1=\nu_1\nu_2$ and $\overleftarrow{\omega}_2=\nu_2\nu_3$.
Since $\omega$ is a palindrome, $\nu_2$ is a palindrome too, and
$\nu_3=\overleftarrow{\nu_1}$.
So $\omega_1$ and $\overleftarrow{\omega_2}$ unite to a new word $\varpi=\nu_1\nu_2\overleftarrow{\nu_1}$, which is a palindrome, and occurs in $\Lambda$ twice at symmetric positions.
Thus the conclusion holds.
\end{proof}

\subsection{The weak type of envelope extension}

\begin{theorem}[]\label{T3.1}\
The factor $\omega$ occurs exactly once in $Env(\omega)$.
\end{theorem}

\begin{proof}
Suppose the factor $\omega$ occurs twice in $Env(\omega)$, denoted by $\omega_1$ and $\omega_2$. Denote $Env(\omega)$ by $E_{i,m}$, $i\in\{1,2\}$ and $m\geq1$.

Case 1. $Env(\omega)=E_{1,m}$.
Since $E_{1,m}=E_{1,m-1}\underline{\delta_{m-1}}E_{1,m-1}$,
both $\omega_1$ and $\omega_2$ contain the middle letter $\delta_{m-1}$.
Otherwise $\omega\prec E_{1,m-1}$, $Env(\omega)\sqsubset E_{1,m-1}$, contradicting $Env(\omega)=E_{1,m}$. By Lemma \ref{L3.1}, we can assume without loss of generality that:
$\omega$ is palindrome; $\omega_1$ and $\omega_2$ occur in $E_{1,m}$ at symmetric positions.

Denote $\omega_1=\mu_1\delta_{m-1}\mu_2$ and $\omega_2=\eta_1\delta_{m-1}\eta_2$.
Since $\omega_1$ and $\omega_2$ occur at symmetric positions,
$|\mu_1|=|\eta_2|>|\mu_2|=|\eta_1|$.
Since $E_{1,m}$ is palindrome, $\omega_2=\overleftarrow{\mu_2}\delta_{m-1}\overleftarrow{\mu_1}$.
Since $\omega$ is palindrome and $|\mu_1|>|\mu_2|$, there exists $\mu_3$ such that
$\mu_1=\overleftarrow{\mu_2}\delta_{m-1}\mu_3$.
This means $\omega=\overleftarrow{\mu_2}\delta_{m-1}\mu_3\delta_{m-1}\mu_2$. Moreover $\mu_3$ is palindrome.
Now, we consider two subcases.

Case 1.1. $|\mu_3|>|\mu_2|$.
Since $\mu_3$ is palindrome, by Property \ref{P3.4}, there exists $n$ such that $\mu_3=E_{1,n}$. Since $\delta_{m-1}\mu_3\delta_{m-1}\prec E_{1,m}$, by
Property \ref{P3.2}, $n\in\{m-2k-1,k=1,2,\ldots,\lfloor\frac{m-2}{2}\rfloor\}$.
Thus
$$\omega=\overleftarrow{\mu_2}\delta_{m-1}\mu_3\delta_{m-1}\mu_2
\prec E_{1,m-3}\delta_{m-1}E_{1,m-3}\delta_{m-1}E_{1,m-3}=E_{2,m-2}.$$
This contradict the hypotheses of $Env(\omega)=E_{1,m}$.

Case 1.2. $|\mu_3|\leq|\mu_2|$.
(1) If there exists $k\geq0$ such that $|\mu_2|+1=k(|\mu_3|+1)$,
$\overleftarrow{\mu_2}=\mu_3(\delta_{m-1}\mu_3)^k$.
Obviously, $\overleftarrow{\mu_2}\delta_{m-1}\mu_3$ is palindrome.
By Property \ref{P3.4}, there exists $n_1$ such that $\overleftarrow{\mu_2}\delta_{m-1}\mu_3=E_{1,n_1}$.
Similarly, $\mu_2$ is palindrome, there exists $n_2<n_1$ such that $\mu_2=E_{1,n_2}$.
Thus
$$\omega=\overleftarrow{\mu_2}\delta_{m-1}\mu_3\underline{\delta_{m-1}}\mu_2
=E_{1,n_1}\underline{\delta_{m-1}}E_{1,n_2}.$$
By Lemma \ref{L1.3},
$\omega=E_{1,n_1}\delta_{m-1}E_{1,n_2}$ does not be a palindrome.

(2) Otherwise, there exists $k\geq1$ and $0<|\mu_4|<|\mu_3|$ such that $\overleftarrow{\mu_2}=\overleftarrow{\mu_4}(\delta_{m-1}\mu_3)^k$.
Thus
$$\omega=\overleftarrow{\mu_2}\delta_{m-1}\mu_3\underline{\delta_{m-1}}\mu_2
=\overleftarrow{\mu_4}(\delta_{m-1}\mu_3)^k\delta_{m-1}\mu_3\underline{\delta_{m-1}}
(\mu_3\delta_{m-1})^k\mu_4
=\overleftarrow{\mu_4}\delta_{m-1}\mu_5\delta_{m-1}\mu_4,$$
where $\mu_5=(\mu_3\delta_{m-1})^k\mu_3
(\delta_{m-1}\mu_3)^k$ is a palindrome with $|\mu_5|>|\mu_4|$.
Using the conclusion in Case 1.1, we get a contradiction of $Env(\omega)=E_{1,m}$.

In summary, Case 1 is impossible.

Case 2. $Env(\omega)=E_{2,m}$.
Since $E_{2,m}=E_{1,m-1}\underline{\delta_{m-1}E_{1,m-1}\delta_{m-1}}E_{1,m-1}$, both $\omega_1$ and $\omega_2$ contain the factor $\delta_{m-1}E_{1,m-1}\delta_{m-1}$ with underline. Otherwise, $\omega\prec E_{1,m-1}\delta_{m-1}E_{1,m-1}=E_{1,m}$,
$Env(\omega)\sqsubset E_{1,m}$, contradicting $Env(\omega)=E_{2,m}$.
Denote $\omega_1=\mu_1\delta_{m-1}\underline{E_{1,m-1}}\delta_{m-1}\mu_2$ and $\omega_2=\eta_1\delta_{m-1}\underline{E_{1,m-1}}\delta_{m-1}\eta_2$.
Obviously, the $E_{1,m-1}$'s in $\omega_1$ and $\omega_2$ are overlapped.
By Theorem \ref{T2.1}, $E_{1,m-1}=A_{m-1}\delta_{m-1}^{-1}$ has only two distinct return words $r_1(E_{1,m-1})=A_{m-1}$ and $r_2(E_{1,m-1})=A_{m-2}$.
Moreover, $E_{1,m-1,p}$ and $E_{1,m-1,p+i}$ are separate for $i\geq2$;
$E_{1,m-1,p}$ and $E_{1,m-1,p+1}$ are separate for $\Theta_1[p]=a$;
$E_{1,m-1,p}$ and $E_{1,m-1,p+1}$ are overlapped for $\Theta_1[p]=b$.
This means the overlap of the $E_{1,m-1}$'s in $\omega_1$ and $\omega_2$ is $r_2^{-1}(E_{1,m-1})E_{1,m-1}=B_{m-2}\delta_{m-1}^{-1}$.
Thus
$$\omega_1=\mu_1\delta_{m-1}A_{m-2}\underline{B_{m-2}\delta_{m-1}^{-1}}\delta_{m-1}\mu_2,~
\omega_2=\mu_1\delta_{m-1}\underline{B_{m-2}\delta_{m-1}^{-1}}\delta_{m}B_{m-2}\mu_2.$$
Notice that the first letter followed the overlap $\underline{B_{m-2}\delta_{m-1}^{-1}}\delta_{m}$ in $\omega_1$ and $\omega_2$
are $\delta_{m-1}$ and $\delta_{m}$, respectively.
This contradict that $\omega_1$ and $\omega_2$ have the same expression.
Thus the conclusion holds.
\end{proof}

Let $\omega$ be a factor, and denote $Env(\omega)$ by $E_{i,m}$. By Theorem \ref{T3.1}, there exist uniquely two words $\mu_1(\omega)$ and $\mu_2(\omega)$
depending only on $\omega$, such that
\begin{equation}\label{E1}
E_{i,m}=\mu_1(\omega)\omega\mu_2(\omega).
\end{equation}

\begin{definition}[Weak type of envelope extension]\label{T3.2}\
Let $\omega$ be a factor, the expression (\ref{E1}) above
is called the weak type of envelope extension of $\omega$.
\end{definition}

\subsection{The extension of $\omega_p$}

\begin{property}[Envelope word inside]\label{P3.6}
For $m>2$,

(1) if $Env(\omega)=E_{1,m}$, $E_{1,m-2}\prec\omega$;
(2) if $Env(\omega)=E_{2,m}$, $E_{1,m-1}\prec\omega$.
\end{property}

\begin{proof} (1) If $Env(\omega)=E_{1,m}$, $\omega$ can't be the factor of $E_{1,m-1}$ or $E_{2,m-2}$.
\setlength{\unitlength}{1mm}
\begin{center}
\begin{picture}(130,28)
\put(32,0){$E_{1,m-1}$}
\put(14,5){$\underbrace{~~~~~~~~~~~~~~~~~~~~~~~~~~~~~~~~~~~~~}$}
\put(50,6){$E_{2,m-2}$}
\put(30,11){$\underbrace{~~~~~~~~~~~~~~~~~~~~~~~~~~~~~~~~~~~~~~~~}$}
\put(0,12){$E_{1,m}=E_{1,m-2}\delta_{m}E_{1,m-3}\delta_{m-1}E_{1,m-3}
\delta_{m-1}E_{1,m-3}\delta_{m-1}E_{1,m-3}\delta_{m}E_{1,m-2}$}
\put(50,15){$\overbrace{~~~~~~~~~~~~~~~~~~~~~~~~~~~~~~~~~~~~~~~~}$}
\put(70,19){$E_{2,m-2}$}
\put(70,22){$\overbrace{~~~~~~~~~~~~~~~~~~~~~~~~~~~~~~~~~~~~~}$}
\put(87,26){$E_{1,m-1}$}
\end{picture}
\end{center}
\centerline{Fig 3.2: $E_{1,m-1}$ and $E_{2,m-2}$ in $E_{1,m}$.}

\vspace{0.2cm}

Consider the three overlaps in the figure above, $\omega$ has three cases:

1) $\omega\succ \delta_{m}E_{1,m-3}\delta_{m-1}E_{1,m-3}\delta_{m-1}$;

2) $\omega\succ \delta_{m-1}E_{1,m-3}\delta_{m-1}E_{1,m-3}\delta_{m-1}$;

3) $\omega\succ \delta_{m-1}E_{1,m-3}\delta_{m-1}E_{1,m-3}\delta_{m}$.

All of them include that $\omega\succ E_{1,m-3}\delta_{m-1}E_{1,m-3}=E_{1,m-2}$.

(2) If $Env(\omega)=E_{2,m}$, $\omega$ can't be the factor of $E_{1,m}$.
\setlength{\unitlength}{1mm}
\begin{center}
\begin{picture}(70,15)
\put(25,0){$E_{1,m}$}
\put(14,5){$\underbrace{~~~~~~~~~~~~~~~~~~~~~~~~}$}
\put(0,6){$E_{2,m}=E_{1,m-1}\delta_{m-1}E_{1,m-1}\delta_{m-1}E_{1,m-1}$}
\put(34,9){$\overbrace{~~~~~~~~~~~~~~~~~~~~~~~~}$}
\put(45,13){$E_{1,m}$}
\end{picture}
\end{center}
\centerline{Fig 3.3: $E_{1,m}$ in $E_{2,m}$.}

\vspace{0.2cm}

Consider the overlap in the figure above,
$\omega\succ \delta_{m-1}E_{1,m-1}\delta_{m-1}\succ E_{1,m-1}$.
\end{proof}

Denote $Env(\omega)=E_{i,m}$. Now we turn to prove that, for each $\omega_p$ there exists  a $q$ such that $\omega\prec E_{i,m,q}$. More precisely, $L(\omega,p)\geq L(E_{i,m},q)$ and $L(\omega,p)+|\omega|\leq L(E_{i,m},q)+|E_{i,m}|$.
In this case, we say that the $\omega_p$ extend to $E_{i,m,q}$.

\vspace{0.2cm}

\noindent$\bullet$~~~\textbf{Case 1.} $\mathbf{Env(\omega)=E_{1,m},~m\geq1.}$

\vspace{0.2cm}

By Theorem \ref{T2.1}, we have
$\{r_p(E_{1,m-2})\}_{p\geq1}$ is $\Theta_1$ over the alphabet
$$\{r_1(E_{1,m-2}),r_2(E_{1,m-2})\}=\{A_{m-2},A_{m-3}\}
=\{E_{1,m-2}\delta_{m},E_{1,m-2}\delta_{m}B_{m-3}^{-1}\},$$
where $E_{1,m-2}=A_{m-2}\delta_{m}^{-1}$.
Since factors with length 3 in $\Theta_1$ are $\{abb,bba,baa,aaa,aab,bab\}$,
$E_{1,m-2}$ occurs in $\mathbb{D}$ has 6 cases, where
$\mathbf{a}=r_1(E_{1,m-2})$, $\mathbf{b}=r_2(E_{1,m-2})$.

(1) $\mathbf{abb}= A_{m-2} E_{1,m-2}\delta_{m}B_{m-3}^{-1} A_{m-3}=A_{m-2}E_{1,m-2}\delta_{m-1}\succ\delta_{m}E_{1,m-2}\delta_{m-1}$;

(2) $\mathbf{bba}= A_{m-3} E_{1,m-2}\delta_{m}B_{m-3}^{-1} A_{m-2}=A_{m-3}E_{1,m-2}\delta_{m-1}B_{m-3}\succ\delta_{m-1}E_{1,m-2}\delta_{m-1}$;

(3) $\mathbf{baa}= A_{m-3} E_{1,m-2}\delta_{m} A_{m-2}\succ\delta_{m-1}E_{1,m-2}\delta_{m}$;

(4) $\mathbf{aaa}= A_{m-2} E_{1,m-2}\delta_{m} A_{m-2}\succ\delta_{m}E_{1,m-2}\delta_{m}$;

(5) $\mathbf{aab}= A_{m-2} E_{1,m-2}\delta_{m} A_{m-3}\succ\delta_{m}E_{1,m-2}\delta_{m}$;

(6) $\mathbf{bab}= A_{m-3} E_{1,m-2}\delta_{m} A_{m-3}\succ\delta_{m-1}E_{1,m-2}\delta_{m}$.

Here we always rewrite the middle letter $\mathbf{a}$ or $\mathbf{b}$ to be an expression with envelope word $E_{1,m-2}$.

\begin{lemma}\label{L3.4}
The factor $\delta_{m}E_{1,m-2}\delta_{m-1}$ occurs in the period-doubling sequence always preceded by the word $E_{1,m-2}=A_{m-2}\delta_{m}^{-1}$
and followed by the word $E_{1,m-1}=A_{1,m-1}\delta_{m-1}^{-1}$.
\end{lemma}

\begin{proof} By the analysis in the beginning of Case 1, $\delta_{m}E_{1,m-2}\delta_{m-1}$ occurs in the period-doubling sequence has only one case: $\mathbf{abb}$. It extends to $\mathbf{\underline{abb}abb}$ or $\mathbf{\underline{abb}aa}$ in $\Theta_1$.

$\mathbf{\underline{abb}abb}=A_{m-2}E_{1,m-2}\delta_{m-1}A_{m-2}A_{m-3}A_{m-3}
=A_{m-2}E_{1,m-2}\delta_{m-1}A_{m-1}$;

$\mathbf{\underline{abb}aa}=A_{m-2}E_{1,m-2}\delta_{m-1}A_{m-2}A_{m-2}
=A_{m-2}E_{1,m-2}\delta_{m-1}B_{m-1}$.

Since $B_{m-1}\delta_{m}^{-1}=A_{m-1}\delta_{m-1}^{-1}$,
both of them have prefix
$A_{m-2}\delta_{m}^{-1}\underline{\delta_{m}E_{1,m-2}\delta_{m-1}} A_{1,m-1}\delta_{m-1}^{-1}$. Here we show the factor $\delta_{m}E_{1,m-2}\delta_{m-1}$ with underline.
This means the conclusion holds.
\end{proof}

\begin{lemma}\label{L3.5}
The factor $\delta_{m-1}E_{1,m-2}\delta_{m-1}$ occurs in the period-doubling sequence always preceded by the word $E_{1,m-2}\delta_{m}E_{1,m-3}=A_{m-2}A_{m-3}\delta_{m-1}^{-1}$
and followed by the word $E_{1,m-3}\delta_{m}E_{1,m-2}=B_{m-3}A_{m-2}\delta_{m}^{-1}$.
\end{lemma}

\begin{proof} By the analysis in the beginning of Case 1, $\delta_{m-1}E_{1,m-2}\delta_{m-1}$ occurs in the period-doubling sequence has only one case: $\mathbf{bba}$. It extends to $\mathbf{a\underline{bba}a}$ or $\mathbf{a\underline{bba}bb}$ in $\Theta_1$.

$\mathbf{a\underline{bba}a}
=A_{m-2}A_{m-3}E_{1,m-2}\delta_{m-1}B_{m-3}A_{m-2}$.

$\mathbf{a\underline{bba}bb}
=A_{m-2}A_{m-3}E_{1,m-2}\delta_{m-1}B_{m-3}A_{m-3}A_{m-3} =A_{m-2}A_{m-3}E_{1,m-2}\delta_{m-1}B_{m-3}B_{m-2}$.

Both of them have prefix $A_{m-2}A_{m-3}\delta_{m-1}^{-1}\underline{\delta_{m-1}E_{1,m-2}\delta_{m-1}} B_{m-3}A_{m-2}\delta_{m}^{-1}$.
Here we show the factor $\delta_{m-1}E_{1,m-2}\delta_{m-1}$ with underline. This means the conclusion holds.
\end{proof}

\begin{lemma}\label{L3.6}
The factor $\delta_{m-1}E_{1,m-2}\delta_{m}$ occurs in the period-doubling sequence always preceded by the word $E_{1,m-1}=A_{m-1}\delta_{m-1}^{-1}$
and followed by the word $E_{1,m-2}=A_{m-2}\delta_{m}^{-1}$.
\end{lemma}

\begin{proof} By the analysis in the beginning of Case 1, $\delta_{m-1}E_{1,m-2}\delta_{m-1}$ occurs in the period-doubling sequence has two cases: $\mathbf{baa}$ and $\mathbf{bab}$.
They extend to $\mathbf{ab\underline{baa}}$ and $\mathbf{ab\underline{bab}b}$ in $\Theta_1$, respectively.

$\mathbf{ab\underline{baa}}=A_{m-2}A_{m-3}A_{m-3}E_{1,m-2}\delta_{m}A_{m-2}
=A_{m-1}E_{1,m-2}\delta_{m}A_{m-2}$;

$\mathbf{ab\underline{bab}b}=A_{m-2}A_{m-3}A_{m-3}E_{1,m-2}\delta_{m}A_{m-3}A_{m-3}
=A_{m-1}E_{1,m-2}\delta_{m}B_{m-2}$.

Since $A_{m-2}\delta_{m}^{-1}=B_{m-2}\delta_{m-1}^{-1}$, both of them have prefix
$A_{m-1}\delta_{m-1}^{-1}\underline{\delta_{m-1}E_{1,m-2}\delta_{m}} A_{m-2}\delta_{m}^{-1}$.
Here we show the factor $\delta_{m-1}E_{1,m-2}\delta_{m}$ with underline. This means the conclusion holds.
\end{proof}

\begin{property}\label{P3.9}
If $Env(\omega)=E_{1,m}$, then each $\omega_p$ can extend to a $E_{1,m}$ in $\mathbb{D}$.
\end{property}

\begin{proof} By Lemma \ref{L2.3}, $\delta_{m}E_{1,m-2}\delta_{m}=\delta_{m}B_{m-2}\not\prec E_{1,m}=A_{m-2}B_{m-2}A_{m-2}A_{m-2}\delta_{m}^{-1}$.
By Figure 3.2,
for $m>2$ and $Env(\omega)=E_{1,m}$, $\omega$ contains one of the three words below:

(1) $\omega\succ\delta_{m}E_{1,m-2}\delta_{m-1}$. By Lemma \ref{L3.4},
$\omega\prec A_{m-2}\delta_{m}^{-1}\delta_{m}E_{1,m-2}\delta_{m-1} A_{1,m-1}\delta_{m-1}^{-1}=E_{1,m}$.

(2) $\omega\succ\delta_{m-1}E_{1,m-2}\delta_{m-1}$. By Lemma \ref{L3.5},
$\omega\prec A_{m-2}A_{m-3}E_{1,m-2}\delta_{m-1} B_{m-3}A_{m-2}\delta_{m}^{-1}
=E_{1,m}$.

(3) $\omega\succ\delta_{m-1}E_{1,m-2}\delta_{m}$. By Lemma \ref{L3.6},
$\omega\prec A_{m-1}\delta_{m-1}^{-1}\delta_{m-1}E_{1,m-2}\delta_{m} A_{m-2}\delta_{m}^{-1}=E_{1,m}$.

In summary, no matter where $\omega$ occurs, it can extend to a $E_{1,m}$ in $\mathbb{D}$.
\end{proof}

\noindent$\bullet$~~~\textbf{Case 2.} $\mathbf{Env(\omega)=E_{2,m},~m\geq1.}$

\vspace{0.2cm}

By Theorem \ref{T2.1}, we have
$\{r_p(E_{1,m-1})\}_{p\geq1}$ in $\mathbb{D}$ is $\Theta_1$ over the alphabet
$$\{r_1(E_{1,m-1}),r_2(E_{1,m-1})\}=\{A_{m-1},A_{m-2}\}
=\{E_{1,m-1}\delta_{m-1},E_{1,m-1}\delta_{m-1}B_{m-2}^{-1}\},$$
where $E_{1,m-1}=A_{m-1}\delta_{m-1}^{-1}$.
Since factors with length 3 in $\Theta_1$ are $\{abb,bba,baa,aaa,aab,bab\}$,
$\mathbf{a}=r_1(E_{1,m-1})$, $\mathbf{b}=r_2(E_{1,m-1})$,
$E_{1,m-1}$ occurs in $\mathbb{D}$ has 6 cases as follow:

(1) $\mathbf{abb}= A_{m-1} E_{1,m-1}\delta_{m-1}B_{m-2}^{-1} A_{m-2}=A_{m-1}E_{1,m-1}\delta_{m}\succ\delta_{m-1}E_{1,m-1}\delta_{m}$;

(2) $\mathbf{bba}= A_{m-2} E_{1,m-1}\delta_{m-1}B_{m-2}^{-1} A_{m-1}=A_{m-2}E_{1,m-1}\delta_{m}B_{m-2}\succ\delta_{m}E_{1,m-1}\delta_{m}$;

(3) $\mathbf{baa}= A_{m-2} E_{1,m-1}\delta_{m-1} A_{m-1}\succ\delta_{m}E_{1,m-1}\delta_{m-1}$;

(4) $\mathbf{aaa}= A_{m-1} E_{1,m-1}\delta_{m-1} A_{m-1}\succ\delta_{m-1}E_{1,m-1}\delta_{m-1}$;

(5) $\mathbf{aab}= A_{m-1} E_{1,m-1}\delta_{m-1} A_{m-2}\succ\delta_{m-1}E_{1,m-1}\delta_{m-1}$;

(6) $\mathbf{bab}= A_{m-2} E_{1,m-1}\delta_{m-1} A_{m-2}\succ\delta_{m}E_{1,m-1}\delta_{m-1}$.

Here we always rewrite the middle letter $\mathbf{a}$ or $\mathbf{b}$ to be an expression with envelope word $E_{1,m-1}$.

\begin{lemma}\label{L3.7}
The word $\delta_{m-1}E_{1,m-1}\delta_{m-1}$ occurs in the period-doubling sequence always preceded and followed by the word $E_{1,m-1}=A_{m-1}\delta_{m-1}^{-1}$.
\end{lemma}

\begin{proof} By the analysis above, $\delta_{m-1}E_{1,m-1}\delta_{m-1}$ occurs in the period-doubling sequence has two cases: $\mathbf{aaa}$ and $\mathbf{aab}$.
Here $\mathbf{\underline{aaa}}=A_{m-1}E_{1,m-1}\delta_{m-1}A_{m-1}$.
Moreover $\mathbf{aab}$ extends to $\mathbf{\underline{aab}b}$ in $\Theta_1$.
$\mathbf{\underline{aab}b}=A_{m-1}E_{1,m-1}\delta_{m-1}A_{m-2}A_{m-2}=A_{m-1}E_{1,m-1}\delta_{m-1}B_{m-1}$.
Since $B_{m-1}\delta_{m}^{-1}=A_{m-1}\delta_{m-1}^{-1}$, both of them have prefix
$A_{m-1}\delta_{m-1}^{-1}\underline{\delta_{m-1}E_{1,m-1}\delta_{m-1}} A_{m-1}\delta_{m-1}^{-1}$.
This means the conclusion holds.
\end{proof}

\begin{property}\label{P3.10}
If $Env(\omega)=E_{2,m}$, then each $\omega_p$ can extend to a $E_{2,m}$ in $\mathbb{D}$.
\end{property}

\begin{proof} By Lemma \ref{L2.1} and \ref{L2.3},
the next three words are not the factors of $E_{2,m}=B_mB_{m-1}\delta_{m}^{-1}$,
$$\delta_{m}E_{1,m-1}\delta_{m}=\delta_{m}B_{m-1},~
\delta_{m-1}E_{1,m-1}\delta_{m}=\delta_{m-1}B_{m-1},~
\delta_{m}E_{1,m-1}\delta_{m-1}=\delta_{m}A_{m-1}.$$
By Figure 3.3, $\delta_{m-1}E_{1,m-1}\delta_{m-1}\prec\omega$.
By Lemma \ref{L3.7}, $\omega\prec A_{m-1}\delta_{m-1}^{-1}\delta_{m-1}E_{1,m-1}\delta_{m-1} A_{m-1}\delta_{m-1}^{-1}
=B_{m}B_{m-1}\delta_{m}^{-1}=E_{2,m}$.
This means no matter where $\omega$ occurs, it can extend to a $E_{1,m}$ in $\mathbb{D}$.
\end{proof}

\subsection{The strong type of envelope extension}

\begin{theorem}\label{T3.7}
$Env(\omega_p)=Env(\omega)_p=E_{i,m,p}$, $i=1,2$ and $p\geq1$.
\end{theorem}

\begin{proof} By Property \ref{P3.9} and \ref{P3.10}, we have for all $p$, there is a $q$ s.t. $E_{i,m,q}=Env(\omega_p)$. By the definition of $Env(\omega)$, $\omega\prec Env(\omega)$. So for all $q$, there is a $p$ s.t. $\omega_p\prec E_{i,m,q}$.
By the weak type of envelope extension, we have (1) if $\omega_{p_1}\prec E_{i,m,q}$ and $\omega_{p_2}\prec E_{i,m,q}$, then $p_1=p_2$; (2) if $\omega_p\prec E_{i,m,q_1}$ and $\omega_p\prec E_{i,m,q_2}$, then $q_1=q_2$.
So $Env(\omega_p)=E_{i,m,q}$ means $p=q$. Thus $Env(\omega_p)=Env(\omega)_p=E_{i,m,p}$.
\end{proof}

Let $\omega$ be a factor, and denote $Env(\omega)$ by $E_{i,m}$. By Theorem \ref{T3.7}, there exist uniquely two words $\mu_1(\omega)$ and $\mu_2(\omega)$
depending only on $\omega$, such that
\begin{equation}\label{E2}
E_{i,m,p}=\mu_1(\omega)\omega_p\mu_2(\omega).
\end{equation}
The $\mu_1(\omega)$ and $\mu_2(\omega)$ have the same expressions with the $\mu_1(\omega)$ and $\mu_2(\omega)$ in expression (\ref{E1}).

\begin{definition}[Strong type of envelope extension]\label{T3.8}\
Let $\omega$ be a factor, the expression (\ref{E2}) above
is called the strong type of envelope extension of $\omega$.
\end{definition}

\section{The return word sequences of general factors in $\mathbb{D}$}

By the strong type of envelope extension, we can extend Theorem \ref{T2.1} and \ref{T2.2} to general factors.

\begin{theorem}[]\label{T3.6}\
(1) When $Env(\omega)=E_{1,m}$,
the return word sequence $\{r_p(\omega)\}_{p\geq1}$ in $\mathbb{D}$ is $\Theta_1$ over the alphabet $\{r_1(\omega),r_2(\omega)\}$.
(2) When $Env(\omega)=E_{2,m}$,
the return word sequence $\{r_p(\omega)\}_{p\geq1}$ in $\mathbb{D}$ is $\Theta_2$ over the alphabet $\{r_1(\omega),r_2(\omega),r_4(\omega)\}$.
\end{theorem}

\begin{property}[]\label{P3.7}
Let $Env(\omega)$ be $E_{i,m}$.
$r_0(\omega)=r_0(E_{i,m})\mu_1(\omega)$,
$r_p(\omega)=\mu_1(\omega)^{-1}r_p(E_{i,m})\mu_1(\omega)$ for $p\geq1$.
\end{property}

\begin{proof} The proof of the property will be easy by the following figure.
\setlength{\unitlength}{0.9mm}
\begin{center}
\begin{picture}(75,15)
\linethickness{1pt}
\put(0,5){\line(1,0){45}}
\linethickness{2pt}
\put(45,5){\line(1,0){20}}
\linethickness{1pt}
\put(65,5){\line(1,0){10}}
\put(30,5){\line(0,1){10}}
\put(75,5){\line(0,1){10}}
\put(0,0){\line(0,1){15}}
\put(45,0){\line(0,1){10}}
\put(65,0){\line(0,1){10}}
\put(36,7){$\mu_1$}
\put(68,7){$\mu_2$}
\put(53,0){$\omega_1$}
\put(50,12){$E_{i,m,1}$}
\put(7,9){$r_0(E_{i,m})$}
\put(17,0){$r_0(\omega)$}
\put(48,13){\vector(-1,0){18}}
\put(62,13){\vector(1,0){13}}
\put(6,10){\vector(-1,0){6}}
\put(24,10){\vector(1,0){6}}
\put(14,1){\vector(-1,0){14}}
\put(30,1){\vector(1,0){15}}
\end{picture}
\end{center}
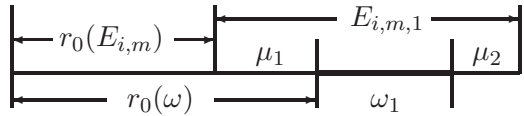
\centerline{Fig. 4.1: The relation among $\omega_1$, $E_{i,m,1}$, $r_0(\omega)$ and $r_0(E_{i,m})$.}
\begin{center}
\begin{picture}(120,22)
\linethickness{2pt}
\put(15,5){\line(1,0){20}}
\put(90,5){\line(1,0){20}}
\linethickness{1pt}
\put(0,5){\line(1,0){15}}
\put(35,5){\line(1,0){55}}
\put(110,5){\line(1,0){10}}
\put(0,5){\line(0,1){15}}
\put(45,5){\line(0,1){10}}
\put(75,5){\line(0,1){15}}
\put(120,5){\line(0,1){10}}
\put(15,0){\line(0,1){10}}
\put(35,5){\line(0,1){5}}
\put(90,0){\line(0,1){10}}
\put(110,0){\line(0,1){10}}
\put(6,7){$\mu_1$}
\put(81,7){$\mu_1$}
\put(38,7){$\mu_2$}
\put(113,7){$\mu_2$}
\put(23,7){$\omega_p$}
\put(96,7){$\omega_{p+1}$}
\put(20,12){$E_{i,m,p}$}
\put(93,12){$E_{i,m,p+1}$}
\put(30,18){$r_p(E_{i,m})$}
\put(45,0){$r_p(\omega)$}
\put(19,13){\vector(-1,0){19}}
\put(32,13){\vector(1,0){13}}
\put(93,13){\vector(-1,0){18}}
\put(108,13){\vector(1,0){12}}
\put(28,19){\vector(-1,0){28}}
\put(47,19){\vector(1,0){28}}
\put(42,1){\vector(-1,0){27}}
\put(58,1){\vector(1,0){32}}
\end{picture}
\end{center}
\centerline{Fig. 4.2: The relation among $\omega_p$, $E_{i,m,p}$, $r_p(\omega)$ and $r_p(E_{i,m})$.}
\end{proof}

The property above can be proved easily by Figure 4.1 and 4.2.
By Property \ref{P2.4}, \ref{P2.5} and \ref{P3.7}, we can give the expressions of $r_p(\omega)$. Since the expressions are complicated, we omit them.

\begin{corollary}[The Lengths of $r_p(\omega)$]\label{C3.2}
Denote $Env(\omega)$ by $E_{i,m}$, let $E_{i,m}=\mu_1(\omega)\omega\mu_2(\omega)$, then

(1) if $i=1$, $|r_0(\omega)|=|\mu_1(\omega)|$, $|r_1(\omega)|=2^m$,
$|r_2(\omega)|=2^{m-1}$;

(2) if $i=2$, $|r_0(\omega)|=2^m+|\mu_1(\omega)|$, $|r_1(\omega)|=2^{m-1}$,
$|r_2(\omega)|=7\times2^{m-1}$,
$|r_4(\omega)|=3\times2^{m-1}$.
\end{corollary}

\vspace{0.5cm}

\noindent\textbf{\large{Acknowledgments}}

\vspace{0.4cm}

The research is supported by the Grant NSFC No.11431007, No.11271223 and No.11371210.

\end{CJK*}
\end{document}